\documentclass{article}

\usepackage{amssymb}

\newcommand{\proof}{{\bf Proof:  }}

\newtheorem{theorem}{Theorem}[section]
\newtheorem{lemma}[theorem]{Lemma}
\newtheorem{definition}[theorem]{Definition}
\newtheorem{proposition}[theorem]{Proposition}
\newtheorem{corollary}[theorem]{Corollary}

\begin{document}

\parindent0pt
\title{\bf Quiver moduli and small desingularizations of some GIT quotients}
\author{Markus Reineke\\ School of Mathematics and Natural Sciences\\ Bergische Universit\"at Wuppertal}

\maketitle
\begin{abstract}We construct small desingularizations of moduli spaces of semistable quiver representations for indivisible dimension vectors using deformations of stabilites and a dimension estimate for nullcones. We apply this construction to several classes of GIT quotients.\end{abstract}

\section{Introduction}\label{s1}

In the introduction of \cite{FGIK}, J. Bernstein is cited: {\it I'd say that if you can compute the Poincar\'e polynomial
for intersection cohomology without a computer then,
probably, there is a small resolution which gives it.}\\[1ex]
For moduli spaces of semistable quiver representations of fixed dimension vector (under certain, rather natural, conditions), the Poincar\'e polynomial for intersection cohomology is identified with the quantized Donaldson-Thomas invariants for quivers with stability in \cite{MeR}; these invariants are defined in terms of a recursive, but nevertheless explicit, formula. It is thus natural, in light of J. Bernstein's principle, to search for small desingularizations of quiver moduli.\\[1ex]
As a first step in this direction, the main result of the present paper, Theorem \ref{maintheorem}, exhibits small desingularizations of moduli spaces of semistable representations for indivisible dimension vectors.\\[1ex]
In fact, the desingularization itself is a moduli space of semistable representations of the same quiver, but with respect to a deformed stability function. The fibres of the desingularization map can be described in terms of moduli spaces of representations for so-called local quivers, essentially using Luna's slice theorem, a technique going back to \cite{LBPSS}. Smallness is then proven using a dimension estimate for nullcones of quiver representations which already played a crucial role in \cite{MeR}.\\[1ex]
Although the indivisibility assumption looks very restrictive from the quiver point of view, our construction nevertheless yields a large class of small resolutions for rather classical GIT quotients, for example moduli spaces of ordered point configurations in projective spaces.\\[2ex]
In Section \ref{s2}, we first review in some detail the necessary background on quiver representations, their geometry and invariant theory, the construction of moduli of semistable representations, and formulas for their (intersection) cohomology. In Section \ref{s3}, the desingularizations are constructed, and their fibres are described. In Section \ref{s4}, after deriving the above mentioned dimension estimate for nullcones, the main result is proven and its cohomological consequences are discussed. Section \ref{s5} explores the generality under which our main result is applicable. In Section \ref{s6}, several example classes of GIT quotients for which small desingularizations can be constructed with the present methods are discussed:  certain determinantal varieties, moduli spaces of point configurations in projective spaces, certain classes of quotients by Levi actions, and the toric quiver moduli of \cite{MoR}.\\[2ex]
{\bf Acknowledgments:} The author would like to thank L. Le Bruyn and S. Mozgovoy for interesting discussions on the material developed here.

\section{Quiver representations and their moduli}\label{recollections}\label{s2}

In this section, we first review basic facts on quivers and their representations, followed by a discussion of the invariant theory of quivers following \cite{LBPSS}. We then review the construction of and basic facts on moduli spaces of (semi-)stable representations of quivers. Finally, we discuss formulas for the (intersection) Betti numbers of quiver moduli.  As a general reference for this material, the reader is referred to e.g. \cite{LBBuch,RModuli}.

\subsection{Quivers and their representations}\label{s21}

Let $Q$ be a finite quiver with (finite) set of vertices $Q_0$ and (finitely many) arrows $\alpha:i\rightarrow j$. The quiver $Q$ is allowed to have oriented cycles (and will typically have in many of the examples below). Let $\Lambda={\bf Z}Q_0$ be the free abelian group generated by the vertices of $Q$, define the ${\bf Q}$-vector space $\Lambda_{\bf Q}={\bf Q}\otimes_{\bf Z}\Lambda$, and let $\Lambda^+={\bf N}Q_0$ be the sub-semigroup of $\Lambda$ generated by $Q_0$. The elements of $\Lambda^+$ are viewed as dimension vectors for $Q$, and are written as ${\bf d}=\sum_{i\in Q_0}d_ii$. On $\Lambda$, we have the Euler form $\langle\_,\_\rangle$ of the quiver, given by
$$\langle{\bf d},{\bf e}\rangle=\sum_{i\in Q_0}d_ie_i-\sum_{\alpha:i\rightarrow j}d_ie_j.$$

We denote by ${\rm rep}_{\bf C}(Q)$ the abelian ${\bf C}$-linear category of finite-dimensional complex representations of $Q$; its objects are thus of the form $$V=((V_i)_{i\in Q_0},(f_\alpha:V_i\rightarrow V_j)_{\alpha:i\rightarrow j})$$
for finite-dimensional ${\bf C}$-vector spaces $V_i$ and ${\bf C}$-linear maps $f_\alpha$. We denote by $${\bf dim} V=\sum_{i\in Q_0}\dim V_i\cdot i\in\Lambda^+$$
the dimension vector of $V$. The Euler form as defined above is then precisely the homological Euler form of ${\rm rep}_{\bf C}(Q)$, that is,
$$\dim{\rm Hom}(V,W)-\dim{\rm Ext}^1(V,W)=\langle{\bf dim} V,{\bf dim} W\rangle$$
for all representations $V$ and $W$, whereas ${\rm Ext}^i(V,W)=0$ for all $i\geq 2$.

\subsection{Geometry and invariant theory of quivers}\label{s22}

Let
$$R_{\bf d}(Q)=\bigoplus_{\alpha:i\rightarrow j}{\rm Hom}_{\bf C}({\bf C}^{d_i},{\bf C}^{d_j})$$ be the variety of complex representations of $Q$ of dimension vector ${\bf d}\in\Lambda^+$, on which the structure group $$G_{\bf d}=\prod_{i\in Q_0}{\rm GL}({\bf C}^{d_i})$$ acts via base change:
$$(g_i)_i\cdot(f_\alpha)_\alpha=(g_jf_\alpha g_i^{-1})_{\alpha:i\rightarrow j}.$$

The $G_{\bf d}$-orbits in $R_{\bf d}(Q)$ thus correspond bijectively to the isomorphism classes of representations of $Q$ of dimension vector ${\bf d}$. The orbit $G_{\bf d}V$ of a (point corresponding to a) representation $V$ in $R_{\bf d}(Q)$ is closed if and only if the representation $V$ is semisimple, by a direct application of the Hilbert criterion. \\[1ex]
The quotient

$$M_{\bf d}^{\rm ssimp}(Q)=R_{\bf d}(Q)//G_{\bf d}={\rm Spec}({\bf C}[R_{\bf d}(Q)]^{G_{\bf d}})$$
is given as the spectrum of the ring of $G_{\bf d}$-invariant polynomial functions on $R_{\bf d}(Q)$. Let $\pi:R_{\bf d}(Q)\rightarrow M_{\bf d}^{\rm ssimp}(Q)$ be the quotient map.\\[1ex]
As the invariant-theoretic quotient parametrizes closed orbits, $M_{\bf d}^{\rm ssimp}(Q)$ is an affine variety parametrizing isomorphism classes of semisimple representations of $Q$ of dimension vector ${\bf d}$. If the quiver $Q$ is acyclic, the only simple representations are the one-dimensional representations $S_i$ supported on a vertex $i\in Q_0$, thus the only semisimple representation of dimension vector ${\bf d}$ is $\bigoplus_iS_i^{d_i}$. Thus $M_{\bf d}^{\rm ssimp}(Q)$ reduces to a single point if $Q$ is acyclic.\\[1ex]
If there exists a simple representation of $Q$ of dimension vector ${\bf d}$ (an efficient numerical criterion for this is derived in \cite{LBPSS}), the variety $M_{\bf d}^{\rm ssimp}(Q)$ has dimension $1-\langle {\bf d},{\bf d}\rangle$.\\[1ex]
There is an additional action of the multiplicative group ${\bf C}^*$ on $R_{\bf d}(Q)$ by rescaling the linear maps representing the arrows, which commutes with the $G_{\bf d}$-action. Thus, this action descends to an action on $M_{\bf d}^{\rm ssimp}(Q)$, which is thus a cone whose vertex corresponds to the representation $\bigoplus_iS_i^{d_i}$ in which all arrows are represented by zero maps. \\[1ex]
Coordinates for $M_{\bf d}^{\rm ssimp}(Q)$ are provided by generators of the invariant ring ${\bf C}[R_{\bf d}(Q)]^{G_{\bf d}}$. It is known that this ring is generated by traces along oriented cycles. More precisely, let $$c:i_0\stackrel{\alpha_1}{\rightarrow}i_1\stackrel{\alpha_2}{\rightarrow}\ldots\stackrel{\alpha_s}{\rightarrow}i_s=i_0$$
be an oriented cycle in $Q$. Then the function
$$t_c(V):={\rm trace}(f_{\alpha_s}\circ\ldots\circ f_{\alpha_1})$$
is $G_{\bf d}$-invariant, and all these functions generate the invariant ring (in fact, it suffices to consider oriented cycles of length at most $(\dim{\bf d})^2$ by \cite{LBPSS}).\\[1ex]
This description of the invariant ring also allows us to determine the zero fibre $\pi^{-1}(0)$ of the quotient map, that is, the nullcone $N_{\bf d}(Q)$ of $R_{\bf d}(Q)$: it consists of all representations which are nilpotent in the sense that every oriented cycle is represented by a nilpotent linear map.\\[1ex] 
More generally, all fibres of the quotient can be determined explicitly using Luna's slice theorem, as explained in \cite{LBPSS}. This will be reviewed in more detail since it will be crucial in the derivation of the main result of this paper.\\[1ex]
Let us first recall the Luna stratification of $M_{\bf d}^{\rm ssimp}(Q)$: a point $V$ of this moduli space corresponds to an isomorphism class of a semisimple representation, again denoted by $V$, of $Q$ of dimension vector ${\bf d}$. Thus $V$ is isomorphic to a direct sum $U_1^{m_1}\oplus\ldots\oplus U_s^{m_s}$ of pairwise non-isomorphic simple representations of. The decomposition type of $V$ is defined as the tuple
$$\xi=({\rm\bf dim} U_1,m_1),\ldots,{\rm\bf dim} U_s,m_s)).$$
Conversely, for a tuple $\xi=(({\bf d}^1,m_1),\ldots,{\bf d}^s,m_s))\in(\Lambda^+\times{\bf N})^s$ such that $\sum_km_k{\bf d}^k={\bf d}$, we define $S_\xi$ as the set of all points of $M_{\bf d}^{\rm ssimp}(Q)$ of decomposition type $\xi$, which is a locally closed subset of $M_{\bf d}^{\rm ssimp}(Q)$ by \cite{LBPSS}. It is known that the quotient map $R_{\bf d}(Q)\rightarrow M_{\bf d}^{\rm ssimp}(Q)$ is \'etale locally trivial over every Luna stratum \cite{LBPSS}. The dense stratum is the locus of simple representations in $M_{\bf d}^{\rm ssimp}(Q)$; it corresponds to the trivial decomposition type $(({\bf d},1))$.\\[1ex]
To a decomposition type $\xi$ as above, we associate the quiver $Q_\xi$ with vertices $i_1,\ldots,i_s$. The number of arrows from $i_k$ to $i_l$ is given as $\delta_{k,l}-\langle {\bf d}^k,{\bf d}^l\rangle$. We define a dimension vector ${\bf d}_\xi$ for $Q_\xi$ by ${\bf d}_\xi=\sum_{k=1}^sm_ki_k$.\\[1ex]
We then have the following description of the fibre over a point $V\simeq\bigoplus_{k}U_k^{m_k}$ in $S_\xi$:

$$\pi^{-1}(V)\simeq G_{\bf d}\times^{G_{{\bf d}_\xi}}N_{{\bf d}_\xi}(Q_\xi).$$

In representation-theoretic terms, this isomorphism corresponds to Ringel simplification \cite{Ringel}: we consider the subcategory $\mathcal{C}$ of ${\rm rep}_{\bf C}(Q)$ of representations admitting a Jordan-H\"older filtration by the simple representations $U_1,\ldots,U_s$. This is an abelian subcategory whose simple objects are the $U_k$. Moreover (see \cite{DengXiao} for a detailed proof), the category $\mathcal{C}$ is equivalent to the subcategory of ${\rm rep}_{\bf C}(Q_\xi)$ of representations admitting a Jordan-H\"older filtration by the simple representations $S_{i_1},\ldots,S_{i_s}$ of $Q_\xi$, that is, to the subcategory of nilpotent representations of $Q_\xi$.

\subsection{Stability}\label{s23}

We denote by $\Lambda^*={\rm Hom}(\Lambda,{\bf Z})$ the group of linear functions on $\Lambda$ with its natural basis elements $i^*$ for $i\in Q_0$. We view an element $\Theta$ of $\Lambda^+$ as a stability for $Q$, and associate to it a slope function $\mu:\Lambda^+\setminus\{0\}\rightarrow{\bf Q}$ via
$$\mu({\bf d})=\frac{\Theta({\bf d})}{\dim {\bf d}},$$
where $\dim\in\Lambda^*$ denotes the function given by $\dim {\bf d}=\sum_{i\in Q_0}d_i$. For a fixed slope $\mu\in{\bf Q}$, define $\Lambda_\mu^+$ as the subset of $\Lambda^+$ of all dimension vectors ${\bf d}\in\Lambda^+$ of slope $\mu$.\\[1ex]
The slope of a non-zero representation $V$ is defined as the slope of its dimension vector, $\mu(V)=\mu({\bf dim}V)$. The representation $V$ is called $\Theta$-semistable if $\mu(U)\leq\mu(V)$ for all non-zero subrepresentations $U\subset V$, and it is called $\Theta$-stable if $\mu(U)<\mu(V)$ for all non-zero proper subrepresentations.\\[1ex]
The semistable representations of a fixed slope $\mu\in {\bf Q}$ form an abelian subcategory ${\rm rep}_{\bf C}^\mu(Q)$, whose simple objects are the stable representations of slope $\mu$. In particular, every object in ${\rm rep}_{\bf C}^\mu(Q)$ thus admits a
relative Jordan-H\"older filtration with stable subquotients. For stable representations $V$ and $W$ of the same slope, we have ${\rm Hom}(V,W)=0$ unless $V$ and $W$ are isomorphic, and ${\rm End}(V)\simeq {\bf C}$. We call a representation polystable of slope $\mu$ if it is a semisimple object in ${\rm rep}_{\bf C}^\mu(Q)$, that is, a direct sum of stable representations of the same slope $\mu$.\\[1ex]
For slopes $\mu,\nu\in{\bf Q}$ such that $\mu<\nu$, we have ${\rm Hom}({\rm rep}_{\bf C}^\nu(Q),{\rm rep}_{\bf C}^\mu(Q))=0$. Every object $V$ of ${\rm rep}_{\bf C}(Q)$ admits a unique Harder-Narasimhan filtration, that is, a filtration
$$0=V_0\subset V_1\subset\ldots\subset V_k=V$$
such that all $V_k/V_{k-1}$ are semistable, and
$$\mu(V_1/V_0)>\mu(V_2/V_1)>\ldots>\mu(V_k/V_{k-1}).$$

We note the following invariance property of stabilities: replacing $\Theta$ by $\widetilde{\Theta}=x\cdot\Theta+y\cdot\dim$ for $x\in{\bf Q}^+$ and $y\in{\bf Z}$ does not change the class of (semi-)stable representations, since $\widetilde{\mu}({\bf d})=\mu(d)+x$ for the associated slope functions. In particular, in considering (semi-)stable representations for a fixed dimension vector ${\bf d}$, we can, and will, always assume $\Theta({\bf d})=0$ without loss of generality.

\subsection{Quiver moduli}\label{s24}

Fix a stability $\Theta\in \Lambda^*$, and assume without loss of generality that $\Theta({\bf d})=0$. Let $R_{\bf d}^{\Theta-\rm sst}(Q)$ be the semistable locus in $R_{\bf d}(Q)$, that is, the (open) subset corresponding to semistable representations,  and let $R_{\bf d}^{\Theta-\rm st}(Q)$ be the stable locus in $R_d(Q)$, again an open subset. We consider the character $\chi:G_{\bf d}\rightarrow{\bf C}^*$ given by
$$\chi((g_i)_{i\in Q_0})=\prod_{i\in Q_0}\det(g_i)^{-\Theta_i}.$$
The ring of $\chi$-semi-invariants is the positively graded ring
$${\bf C}[R_{\bf d}(Q)]^{G_{\bf d}}_{\chi}=\bigoplus_{n\geq 0}{\bf C}[R_{\bf d}(Q)]^{G_{\bf d},\chi^n},$$
where ${\bf C}[R_{\bf d}(Q)]^{G_{\bf d},\chi^n}$ consists of polynomial functions $f\in{\bf C}[R_{\bf d}]$ such that
$$f(gV)=\chi(g)^nf(V)\mbox{ for all }g\in G_{\bf d}, V\in R_{\bf d}(Q).$$
We define the moduli space of semistable representations of $Q$ of dimension vector ${\bf d}$ as
$$M_{\bf d}^{\Theta-\rm sst}(Q)=R_{\bf d}^{\rm sst}//G_{\bf d}={\rm Proj}({\bf C}[R_{\bf d}(Q)]^{G_{\bf d}}_\chi).$$

On the open subset $R_{\bf d}^{\Theta-\rm st}$, the quotient map is a geometric quotient, so the variety $M_{\bf d}^{\Theta-\rm sst}(Q)$ contains an open subset $$M_{\bf d}^{\Theta-\rm st}(Q)=R_{\bf d}^{\Theta-\rm st}/G_{\bf d},$$ the moduli space of stable representations of $Q$ of dimension vector ${\bf d}$. More precisely, the quotient
$$R_{\bf d}^{\Theta-\rm st}(Q)\rightarrow M_{\bf d}^{\Theta-\rm st}(Q)$$
is a $PG_{\bf d}$-principal bundle, where $PG_{\bf d}$ is the quotient of $G_{\bf d}$ by the diagonally embedded scalars ${\bf C}^*$ (which of course act trivially on $R_{\bf d}(Q)$).\\[1ex]
The points of the moduli space $M_{\bf d}^{\Theta-\rm st}(Q)$ parametrize isomorphism classes of stable representations of $Q$ of dimension vector ${\bf d}$, whereas the points of the moduli space $M_{\bf d}^{\Theta-\rm sst}(Q)$  parametrize isomorphism classes of polystable representations of $Q$ of dimension vector ${\bf d}$ (since these correspond to the closed $G_{\bf d}$-orbits in $R_{\bf d}^{\Theta-\rm sst}(Q)$).  Whereas $M_{\bf d}^{\Theta-\rm st}(Q)$ is a smooth variety of dimension $1-\langle{\bf d},{\bf d}\rangle$ if non-empty, the variety $M_{\bf d}^{\Theta-\rm sst}(Q)$ is typically singular.\\[1ex] 
In the special case $\Theta=0$, every representation is semistable, and the stable representations are precisely the simple ones, thus $M_{\bf d}^{0-\rm sst}(Q)=M_{\bf d}^{\rm ssimp}(Q)$.\\[1ex]
By definition, we have a natural projective morphism
$$p:M_{\bf d}^{\Theta-\rm sst}(Q)\rightarrow M_{\bf d}^{\rm ssimp}(Q);$$
in particular, $M_{\bf d}^{\Theta-\rm sst}(Q)$ is projective if $Q$ is acyclic. In the general case, we define
$$M_{\bf d}^{\Theta-\rm sst,nilp}(Q)=p^{-1}(0)$$
as the moduli space of nilpotent $\Theta$-semistable representations of $Q$ of dimension vector ${\bf d}$. This is a projective, but typically singular, variety; alternatively, it can be realized as
$$M_{\bf d}^{\Theta-\rm sst,nilp}(Q)=(N_{\bf d}(Q)\cap R_{\bf d}^{\Theta-\rm sst}(Q))//G_{\bf d}.$$

We call a dimension vector ${\bf d}\in\Lambda^+$ 
\begin{itemize}
\item indivisible if $\gcd(d_i\, :\, i\in Q_0)=1$, and
\item $\Theta$-coprime if $\Theta({\bf e})\not=\Theta({\bf d})$ for all $0\not={\bf e}\lneqq{\bf d}$.
\end{itemize}

\begin{lemma} If ${\bf d}$ is $\Theta$-coprime, it is indivisible. If ${\bf d}$ is indivisible, then it is $\Theta$-coprime for sufficiently generic $\Theta$.
\end{lemma}

\proof Suppose ${\bf d}=k{\bf e}$ for some $k\geq 2$. Then $0=\Theta({\bf d})=k\Theta({\bf e})$, proving the first claim. For ${\bf d}$ indivisible, the conditions $\Theta({\bf e})=0$ for $0\not={\bf e}\lneqq{\bf d}$ cut out finitely many non-trivial hyperplanes in the space of all $\Theta$ such that $\Theta({\bf d})=0$, thus a choice of $\Theta$ in the complement of these hyperplanes defines a stability for which ${\bf d}$ is $\Theta$-coprime.\\[1ex]
If ${\bf d}$ is $\Theta$-coprime, we thus have $R_{\bf d}^{\Theta-\rm st}(Q)=R_{\bf d}^{\Theta-\rm sst}(Q)$, thus $M_{\bf d}^{\Theta-\rm st}(Q)=M_{\bf d}^{\Theta-\rm sst}(Q)$. Thus $M_{\bf d}^{\Theta-\rm sst}(Q)$ is smooth and projective, of dimension $1-\langle{\bf d},{\bf d}\rangle$, over the affine variety $M_{\bf d}^{\rm ssimp}(Q)$.

\subsection{Cohomology and DT invariants}\label{s25}

We continue to use the setup of the previous section, so let $Q$ be an arbitrary finite quiver, $\Theta$ a stability for $Q$, and ${\bf d}$ a dimension vector.\\[1ex]
We define rational functions $p_{\bf d}(q)\in{\bf Q}(q)$ by
$$p_d(q)=\sum_{{\bf d}^*}(-1)^{s-1}q^{-\sum_{k\leq l}\langle{\bf d}^l,{\bf d}^k\rangle}\prod_{k=1}^s\prod_{i\in Q_0}\prod_{j=1}^{d^k_i}(1-q^{-j})^{-1},$$
where the sum ranges over all ordered decompositions ${\bf d}={\bf d}^1+\ldots+{\bf d}^s$ of ${\bf d}$ into non-zero dimension vectors such that
$$\mu({\bf d}^1+\ldots+{\bf d}^k)>\mu({\bf d})$$
for all $k<s$.\\[1ex]
With the aid of these functions (arising from a resolved Harder-Narasimhan recursion, see \cite{RHN}), we can determine the Betti numbers in cohomology with compact support of the moduli spaces $M_{\bf d}^{\Theta-\rm sst}(Q)$ in the $\Theta$-coprime case:

\begin{theorem}\label{bettiformula} If ${\bf d}$ is $\Theta$-coprime, we have
$$\sum_i\dim H^i_c(M_{\bf d}^{\Theta-\rm sst}(Q),{\bf Q})(-q^{1/2})^i=(q-1)p_{\bf d}(q).$$
\end{theorem}

In the case of acyclic quivers, this is proved in \cite{RHN}; this restriction is subsequently removed in \cite{ER}.\\[1ex]
In case ${\bf d}$ is not $\Theta$-coprime, this theorem no longer holds true: the rational function $(q-1)p_{\bf d}(q)$ is no longer a polynomial in $q$, and we have seen above that the moduli space $M_{\bf d}^{\Theta-\rm sst}(Q)$ is typically singular, thus its singular cohomology cannot be assumed to be well-behaved.\\[1ex]
However, under a mild restriction (which will also play a distinguished role in the derivation of the main result of this paper), the intersection cohomology of the moduli space is determined using variants of the functions $p_{\bf d}(Q)$ in \cite{MeR}. We briefly review the resulting formula: \\[1ex]
Consider the complete commutative local ring ${\bf Q}(q^{1/2})[[\Lambda^+_\mu]]$ with topological basis $t^{\bf d}$ for ${\bf d}\in\Lambda^+_\mu\cup\{0\}$ and multiplication $t^{\bf d}t^{\bf e}=t^{{\bf d}+{\bf e}}$, with its maximal ideal $\mathfrak{m}$. Define a function ${\rm Exp}:\mathfrak{m}\rightarrow 1+\mathfrak{m}\subset {\bf Q}(q)[[\Lambda^+_\mu]]$, the plethystic exponential, by
$${\rm Exp}(q^it^{\bf d})=\frac{1}{1-q^it^{\bf d}}\mbox{ for }i\in\frac{1}{2}{\bf Z}, {\bf d}\in\Lambda^+_\mu\mbox{, and }{\rm Exp}(f+g)={\rm Exp}(f)\cdot{\rm Exp}(g).$$
We can then define rational functions ${\rm DT}_{\bf d}^{\Theta}(Q)$, the $q$-Donaldson-Thomas invariants of $Q$ for slope $\mu$, by
$$1+\sum_{{\bf d}\in\Lambda^+_\mu}(-q^{1/2})^{\langle{\bf d},{\bf d}\rangle}p_{\bf d}(q)t^{\bf d}={\rm Exp}(\frac{1}{q^{-1/2}-q^{1/2}}\sum_{{\bf d}\in\Lambda^+_\mu}{\rm DT}_{\bf d}^{\Theta}(Q)t^{\bf d}).$$

Then the main result of \cite{MeR} is the following:

\begin{theorem}\label{intersectionbettiformula} Assume that the restriction of $\langle\_,\_\rangle$ to $\Lambda^+_\mu$ is symmetric, and that $M_{\bf d}^{\Theta-\rm sst}(Q)\not=\emptyset$. Then the following formula for the Betti numbers of $M_{\bf d}^{\Theta-\rm sst}(Q)$ in compactly supported intersection cohomology holds:

$$\sum_i\dim{\rm IH}^i_c(M_d^{\rm sst}(Q),{\bf Q})(-q^{1/2})^i=(-q^{1/2})^{1-\langle{\bf d},{\bf d}\rangle}{\rm DT}_{\bf d}^{\Theta}(Q).$$
\end{theorem}

\section{Construction of desingularizations}\label{s3}

In this section, we construct desingularizations of moduli spaces of semistable representations using the concept of a generic deformation of a stability. We describe the fibres of the desingularization in terms of moduli spaces of representations of local quivers using methods of \cite{ALB}.

\begin{definition} Let ${\bf d}$ be a dimension vector for $Q$. Let $\Theta$ be a stability for $Q$ such that $\Theta({\bf d})=0$. A stability $\Theta'$ for $Q$ is called a generic deformation of $\Theta$ with respect to the dimension vector ${\bf d}$ if the following holds for all $0\not={\bf e}\lneqq{\bf d}$:
\begin{enumerate}
\item If $\Theta({\bf e})<0$, then $\Theta'({\bf e})<0$,
\item if $\Theta'({\bf e})\leq 0$, then $\Theta({\bf e})\leq 0$,
\item ${\bf d}$ is $\Theta'$-coprime.
\end{enumerate}
\end{definition}

Note that the last condition already implies that ${\bf d}$ is indivisible. The definition of a generic deformation immediately implies the following:

\begin{lemma} Assume that $\Theta'$ is a generic deformation of $\Theta$ with respect to ${\bf d}$, and let $V$ be a representation of $Q$ of dimension vector ${\bf d}$. Then we have the following implications:
\begin{itemize}
\item If $V$ is $\Theta$-stable, then it is $\Theta'$-stable,
\item $V$ is $\Theta'$-stable if and only if it is $\Theta'$-semistable,
\item if $V$ is $\Theta'$-semistable, it is $\Theta$-semistable.
\end{itemize}
\end{lemma}

An explicit construction of a generic deformation of a given stability $\Theta$ for an indivisible dimension vector ${\bf d}$ will be given in Section \ref{s5}.

\begin{proposition}\label{desing} Let ${\bf d}$ be dimension vector for $Q$, let $\Theta$ be a stability for $Q$ such that $\Theta({\bf d})=0$, and let $\Theta'$ be a generic deformation of $\Theta$ with respect to ${\bf d}$. Assume that $M_{\bf d}^{\Theta-\rm st}(Q)\not=\emptyset$.  Then there exists a desingularization
$$p:M_{\bf d}^{\Theta'-\rm sst}(Q)\rightarrow M_{\bf d}^{\Theta-\rm sst}(Q).$$
\end{proposition}

\proof Since a $\Theta'$-semistable representation is already $\Theta$-semistable by the previous lemma, we have an open inclusion $$R_{\bf d}^{\Theta'-\rm sst}(Q)\subset R_{\bf d}^{\Theta-\rm sst}(Q).$$ Passing to $G_{\bf d}$-quotients, it induces a map $$p:M_{\bf d}^{\Theta'-\rm sst}(Q)\rightarrow M_{\bf d}^{\Theta-\rm sst}(Q).$$
This map commutes with the natural projections to $M_{\bf d}^{\rm ssimp}(Q)$, thus we have a commuting triangle
$$\begin{array}{ccc} M_{\bf d}^{\Theta'-\rm sst}(Q)&\stackrel{p}{\rightarrow}&M_{\bf d}^{\Theta-\rm sst}(Q)\\
&\searrow&\downarrow\\
&&M_{\bf d}^{\rm ssimp}(Q).\end{array}$$
Since the vertical maps are projective, we conclude that $p$ is proper.\\[1ex]
Again by the previous lemma, we have a chain of open embeddings
$$R_{\bf d}^{\Theta-\rm st}(Q)\subset R_{\bf d}^{\Theta'-\rm st}(Q)=R_{\bf d}^{\Theta'-\rm sst}(Q)\subset R_{\bf d}^{\Theta-\rm sst}(Q).$$
Taking $G_{\bf d}$-quotients, we thus have a chain of maps
$$M_{\bf d}^{\Theta-\rm st}(Q)\subset M_{\bf d}^{\Theta'-\rm st}(Q)=M_{\bf d}^{\Theta'-\rm sst}(Q)\stackrel{p}{\rightarrow}M_{\bf d}^{\Theta-\rm sst}(Q).$$
The first map in this chain is an open inclusion since the $G_{\bf d}$-quotient is geometric on stable loci. The composition of this open inclusion with $p:M_{\bf d}^{\Theta'-sst}(Q)\rightarrow M_{\bf d}^{\Theta-\rm sst}(Q)$ equals the open inclusion
$$M_{\bf d}^{\Theta-\rm st}(Q)\subset M_{\bf d}^{\Theta-\rm sst}(Q).$$
We conclude that $p$ is an isomorphism over $M_{\bf d}^{\rm st}(Q)$, and in particular birational in case $M_{\bf d}^{\Theta-\rm st}(Q)\not=\emptyset$.  Moreover, $M_{\bf d}^{\Theta'-\rm sst}(Q)=M_{\bf d}^{\Theta'-\rm st}(Q)$ is smooth , thus $p$ is a desingularization, finishing the proof.\\[2ex]
To determine the fibres of the above desingularization, we first recall a generalization of the Luna stratification and of the description of fibres of quotient maps following \cite{ALB}.\\[1ex]
The points of the moduli space $M_{\bf d}^{\Theta-\rm sst}(Q)$ parametrize isomorphism classes of $\Theta$-polystable representations $V$ of $Q$ of dimension vector ${\bf d}$. By definition, this means that such a representation $V$ is isomorphic to a direct sum $U_1^{m_1}\oplus\ldots\oplus U_k^{m_k}$ of pairwise non-isomorphic stable representations of slope $\mu({\bf d})$. The decomposition type of $V$ is the tuple
$$\xi=(({\rm\bf dim} U_1,m_1),\ldots,{\rm\bf dim} U_s,m_s)).$$
We define $S_\xi$ as the set of all points of $M_{\bf d}^{\Theta-\rm sst}(Q)$ of decomposition type $\xi$, which is a locally closed subset of $M_{\bf d}^{\Theta-\rm sst}(Q)$. It is known \cite{ALB} that the quotient map $R_{\bf d}^{\Theta-\rm sst}(Q)\rightarrow M_{\bf d}^{\Theta-\rm sst}(Q)$ is \'etale locally trivial over every Luna stratum. The dense stratum is $M_{\bf d}^{\Theta-\rm st}(Q)$; it corresponds to the trivial decomposition type $(({\bf d},1))$.\\[1ex]
To a decomposition type $\xi$ as above, we associate the quiver $Q_\xi$ with vertices $i_1,\ldots,i_s$. The number of arrows from $i_k$ to $i_l$ is given as $\delta_{k,l}-\langle {\bf d}^k,{\bf d}^l\rangle$. We define a dimension vector ${\bf d}_\xi$ for $Q_\xi$ by ${\bf d}_\xi=\sum_{k=1}^sm_ki_k$.\\[1ex]
Then we have the following, again by \cite{ALB}:

\begin{theorem} Let $\pi:R_{\bf d}^{\Theta-\rm sst}(Q)\rightarrow M_{\bf d}^{\Theta-\rm sst}(Q)$ be the quotient morphism, and let $V$ be a point in $S_\xi$. Then
$$\pi^{-1}(V)\simeq G_{\bf d}\times^{G_{{\bf d}_\xi}}N_{{\bf d}_\xi}(Q_\xi).$$
\end{theorem}

Given a generic deformation $\Theta'$ of $\Theta$, we define a stability $\Theta_\xi$ for $Q_\xi$ by $$\Theta_\xi(i_k)=\Theta'(d^k)$$ for all $k=1,\ldots,s$. Note that, by definition of $\Theta_\xi$ and a direct calculation, the dimension vector ${\bf d}_\xi$ is $\Theta_\xi$-coprime since ${\bf d}$ is $\Theta'$-coprime.

\begin{theorem}\label{albiso} Over a point $V$ in the Luna stratum $S_\xi\subset M_{\bf d}^{\Theta-\rm sst}(Q)$, the fibre $p^{-1}(V)$ is isomorphic to the moduli space of nilpotent $\Theta_\xi$-semistable representation of $Q_\xi$ of dimension vector ${\bf d}_\xi$:
$$p^{-1}(x)\simeq M_{{\bf d}_\xi}^{\Theta_\xi-\rm sst, nilp}(Q_\xi).$$
\end{theorem}

\proof As in Section \ref{s22}, the isomorphism
$$\pi^{-1}(V)\simeq G_{\bf d}\times^{G_{{\bf d}_\xi}}N_{{\bf d}_\xi}(Q_\xi)$$
of Theorem \ref{albiso} admits a representation-theoretic interpretation using Ringel simplification: we consider the subcategory $\mathcal{C}^\Theta$ of ${\rm rep}_{\bf C}^{0}(Q)$ of semistable representations of slope $0$ admitting a relative Jordan-H\"older filtration by the stable representations $U_1,\ldots,U_s$. This is an abelian subcategory whose simple objects are the $U_k$. Moreover, the category $\mathcal{C}^\Theta$ is equivalent to the subcategory of ${\rm rep}_{\bf C}(Q_\xi)$ of representations admitting a Jordan-H\"older filtration by the simple representations $S_{i_1},\ldots,S_{i_s}$ of $Q_\xi$, that is, to the subcategory of   nilpotent representations of $Q_\xi$.\\[1ex]
Assume that a $\Theta$-semistable representation $W$ of $Q$ dimension vector ${\bf d}$  corresponds to a nilpotent representation $X$ of $Q_\xi$ of dimension vector ${\bf d}_\xi$ under this equivalence. We claim that $W$ is $\Theta'$-semistable if and only if $X$ is $\Theta_\xi$-semistable.\\[1ex]
So assume that $W$ is $\Theta'$-semistable, and let $U_X\subset X$ be a subrepresentation. Under the above equivalence, it corresponds to a semistable subrepresentation $U_W\subset W$ of slope $0$, thus $U_W$ admits a relative Jordan-Hoelder filtration by the $U_k$, say with multiplicities $m_k'=({\bf dim}U_X)_{i_k}$. Since $V$ is $\Theta'$-semistable, we thus have
$$0\geq \Theta'({\bf dim} U_W)=\sum_{k=1}^s m_l'\Theta'({\bf dim} U_k)=\sum_{k=1}^s m_k'\Theta_\xi(i_k)=\Theta_\xi({\bf dim} U_X).$$

Conversely, assume that $X$ is $\Theta_\xi$-semistable, and let $U_W\subset W$ be a subrepresentation; we have to prove that $\Theta'({\bf dim}U_W)\leq 0$. Since $W$ is $\Theta$-semistable, we know that $\Theta({\bf dim} U_W)\leq 0$. If $\Theta({\bf dim}U_W)<0$, then $\Theta'({\bf dim}U_W)<0$ since $\Theta'$ is a generic deformation of $\Theta$, so there is nothing to prove. So assume that $\Theta({\bf dim}U_W)=0$. Thus $U_W$ is $\Theta$-semistable of slope $0$, thus admits a relative Jordan-Hoelder filtration by the stables $U_k$, say with multiplicities $m_k'$. Again under the above equivalence $U_W$ thus corresponds to a subrepresentation $U_X$ of $X$, and as above, we can conclude $\Theta'({\bf dim}U_W)=\Theta_\xi({\bf dim}U_X)\leq 0$, finishing the proof of the claim.\\[1ex]
In geometric terms, we have thus proved that the isomorphism

$$\pi^{-1}(V)\simeq G_{\bf d}\times^{G_{{\bf d}_\xi}}N_{{\bf d}_\xi}(Q_\xi)$$
restricts to an isomorphism of semistable loci

$$\pi^{-1}(V)\cap R_{\bf d}^{\Theta'-\rm sst}(Q)\simeq G_{\bf d}\times^{G_{{\bf d}_\xi}}(N_{{\bf d}_\xi}(Q_\xi)\cap R_{{\bf d}_\xi}^{\Theta_\xi-\rm sst}(Q_\xi)).$$

We consider the commutative square
$$\begin{array}{ccc} R_{\bf d}^{\Theta'-\rm sst}(Q)&\rightarrow&R_{\bf d}^{\Theta-\rm sst}(Q)\\ \pi'\downarrow&&\downarrow \pi\\
M_{\bf d}^{\Theta'-\rm sst}(Q)&\stackrel{p}{\rightarrow}& M_{\bf d}^{\Theta-\rm sst}(Q),\end{array}$$
with the vertical maps being $G_{\bf d}$-quotients, the upper horizonal map being the open inclusion, and the lower horizontal map being the desingularization map $p$. Using this diagram and the previous identification of semistable loci, we can compute the fibre $p^{-1}(V)$ as:
\begin{eqnarray*}
p^{-1}(V)&\simeq&\pi'^{-1}(p^{-1}(V))//G_{\bf d}\\
&\simeq& (\pi^{-1}(V)\cap R_{\bf d}^{\Theta'-\rm sst}(Q))//G_{\bf d}\\
&\simeq&(G_{\bf d}\times^{G_{{\bf d}_\xi}}(N_{{\bf d}_\xi}(Q_\xi)\cap R_{{\bf d}_\xi}^{\Theta_\xi-\rm sst})(Q))//G_{\bf d}\\
&\simeq& (N_{{\bf d}_\xi}(Q_\xi)\cap R_{{\bf d}_\xi}^{\Theta_\xi-\rm sst}(Q))//G_\xi\\
&\simeq& M_{{\bf d}_\xi}^{\Theta_\xi-\rm sst,nilp}(Q_\xi),\end{eqnarray*}
proving the theorem.

\section{Smallness}\label{s4}

The aim of this section is to prove that the desingularization map constructed in Section \ref{s41} is small under the additional assumption that the restriction of the Euler form to the kernel of the stability is symmetric. We first recall in detail a dimension estimate for the nullcone of representations of a symmetric quiver from \cite{MeR}, and then use the description of the fibres of $p$ from the previous section to prove smallness.

\subsection{A dimension estimate for moduli of semistable nilpotent representations}\label{s41}

The main observation of this section is that, under the assumption of $Q$ being symmetric and ${\bf d}$ being $\Theta$-coprime, there is an effective estimate for the dimension of $M_{\bf d}^{\Theta-\rm sst,nilp}(Q)$. We first estimate the dimension of the nullcone following \cite{MeR}:

\begin{theorem}\label{nilpotent stack dimension} If $Q$ is symmetric, we have
$$\dim N_{\bf d}(Q)-\dim G_{\bf d}\leq -\frac{1}{2}\langle {\bf d},{\bf d}\rangle+\frac{1}{2}\sum_{i\in Q_0}\langle i,i\rangle d_i-\dim {\bf d}.$$
\end{theorem}

\begin{proof} Fix a decomposition ${\bf d}={\bf d}^1+\ldots+{\bf d}^s$, denoted by ${\bf d}^*$, and fix a flag $$F^*=F^0\subset F^1\subset\ldots\subset F^s=\bigoplus_{i\in Q_0}{\bf C}^{d_i}$$ of $Q_0$-graded subspaces $F^k=\bigoplus_{i\in Q_0}F^k_i$ such that ${\bf dim}F^k/F^{k-1}={\bf d}^k$ for all $k=1,\ldots,s$.\\[1ex]
Let $Z_{{\bf d}^*}$ be the closed subvariety of $R_{\bf d}(Q)$ consisting of all representations $V=(f_\alpha)_\alpha$ of $Q$ on  $\bigoplus_{i\in Q_0}{\bf C}^{d_i}$ which are compatible with $F^*$, in the sense that $$f_\alpha(F^k_i)\subset F^{k-1}_j\mbox{ for all }\alpha:i\rightarrow j\mbox{ and all }k=1,\ldots,s.$$
Let $P_{{\bf d}^*}$ be the parabolic subgroup of $G_{\bf d}$ of all automorphisms $(g_i)_i$ respecting $F^*$, in the sense that $g_i(F^k_i)\subset F^k_i$ for all $i$ and all $k$.\\[1ex]
We call ${\bf d}^*$ thin if $\dim {\bf d}^k=1$ for all $k$. By definition (that is, by the Hilbert criterion) we have $$N_{\bf d}(Q)=\bigcup_{{\bf d}^*}G_{\bf d}Z_{{\bf d}^*},$$where the union ranges over all thin decompositions ${\bf d}^*$ of ${\bf d}$. Moreover, we have a proper surjective map $$G_{\bf d}\times^{P_{{\bf d}^*}}Z_{{\bf d}^*}\rightarrow G_{\bf d}Z_{{\bf d}^*}.$$ Using this map, the dimension of $G_{\bf d}Z_{{\bf d}^*}$ is easily estimated as
$$\dim G_{\bf d}Z_{{\bf d}^*}\leq\dim G_{\bf d}-\dim P_{{\bf d}^*}+\dim Z_{{\bf d}^*}\leq \dim G_d-\sum_{k<l}\langle d^l,d^k\rangle-\sum_i\sum_k(d^k_i)^2.$$

Since $Q$ is symmetric, we can rewrite
$$\sum_{k<l}\langle {\bf d}^l,{\bf d}^k\rangle=\frac{1}{2}\langle {\bf d},{\bf d}\rangle-\frac{1}{2}\sum_k\langle {\bf d}^k,{\bf d}^k\rangle.$$
All ${\bf d}^k$ being one-dimensional, we can easily rewrite
$$\sum_i\sum_k(d^k_i)^2=\dim {\bf d},\;\;\;
\sum_k\langle {\bf d}^k,{\bf d}^k\rangle=\sum_i\langle i,i\rangle d_i.$$
Thus we can rewrite the above estimate as
$$\dim G_{\bf d}Z_{{\bf d}^*}\leq-\frac{1}{2}\langle{\bf d},{\bf d}\rangle+\frac{1}{2}\sum_i\langle i,i\rangle d_i-\dim{\bf d}.$$
Since $N_{\bf d}(Q)$ is the union of all such $G_{\bf d}Z_{{\bf d}^*}$, the theorem follows.
\end{proof}

\begin{corollary}\label{cor} If $Q$ is symmetric and ${\bf d}$ is $\Theta$-coprime, we have
$$\dim M_{\bf d}^{\Theta-\rm sst,nilp}(Q)\leq -\frac{1}{2}\langle{\bf d},{\bf d}\rangle+\frac{1}{2}\sum_i\langle i,i\rangle d_i-\dim {\bf d}+1.$$
\end{corollary}

\begin{proof} The quotient $$R_{\bf d}^{\Theta-\rm sst}(Q)\rightarrow M_{\bf d}^{\Theta-\rm sst}(Q)$$ being geometric by coprimality, that is, being a $PG_{\bf d}$-principal bundle, the same holds for the restriction to the nilpotent locus $$N_{\bf d}(Q)\cap R_{\bf d}^{\Theta-\rm sst}(Q)\rightarrow M_{\bf d}^{\Theta-\rm sst,nilp}(Q).$$
Thus $$\dim M_{\bf d}^{\Theta-\rm sst,nilp}(Q)=\dim N_{\bf d}(Q)-\dim G_{\bf d}+1.$$
Using the estimate of Theorem \ref{nilpotent stack dimension}, the statement follows.
\end{proof}

\subsection{Main theorem}\label{s42}

Recall that a desingularization map $f:Y\rightarrow X$ of an irreducible variety $X$ is called small if there exists a decomposition of $X$ into finitely many locally closed strata $X_i$, over which $f$ is \'etale locally trivial, such that for all $x\in X_i$, we have the estimate
$$\dim f^{-1}(x)\leq\frac{1}{2}{\rm codim}_XX_i,$$
with equality only for the dense stratum.\\[1ex]
We can now prove the main result.

\begin{theorem}\label{maintheorem} Assume that a quiver $Q$, a dimension vector ${\bf d}$ and stabilities $\Theta$ and $\Theta'$ are given such that:
\begin{enumerate}
\item $M_{\bf d}^{\Theta-\rm st}(Q)\not=\emptyset$,
\item $\Theta({\bf d})=0$, and $\Theta'$ is a generic deformation of $\Theta$,
\item the restriction of $\langle\_,\_\rangle$ to ${\rm Ker}(\Theta)$ is symmetric.
\end{enumerate}

Then the desingularisation $p:M_{\bf d}^{\Theta'-\rm sst}(Q)\rightarrow M_{\bf d}^{\Theta-\rm sst}(Q)$ constructed in  is small.
\end{theorem}

\begin{proof} Let $V$ be a point in the Luna stratum $S_\xi$ of $M_{\bf d}^{\Theta-\rm sst}(Q)$, thus $$p^{-1}(V)\simeq M_{{\bf d}_\xi}^{\Theta_\xi-\rm sst,nilp}(Q_\xi).$$
Under the assumptions of the theorem, the conditions of Corollary \ref{cor} are fulfilled (that is, $Q_\xi$ is symmetric, and ${\bf d}_\xi$ is $\Theta_\xi$-coprime), thus we have the estimate
$$\dim p^{-1}(V)\leq -\frac{1}{2}\langle{\bf d}_\xi,{\bf d}_\xi\rangle_{Q_\xi}+\frac{1}{2}\sum_{k=1}^s\langle i_k,i_k\rangle_{Q_\xi}({\bf d}_\xi)_{i_k}-\dim {\bf d}_\xi+1.$$
Using the definition of $Q_\xi$ and ${\bf d}_\xi$, this can be rewritten in terms of $Q$, ${\bf d}$ and $\xi=(({\bf d}^k,m_k))_k$ as
$$\dim p^{-1}(V)\leq -\frac{1}{2}\langle{\bf d},{\bf d}\rangle+\frac{1}{2}\sum_{k=1}^s\langle {\bf d}^k,{\bf d}^k\rangle m_k-\sum_{k=1}^sm_k+1.$$
We reformulate this estimate as
$$\dim p^{-1}(V)\leq -\frac{1}{2}\langle{\bf d},{\bf d}\rangle-\frac{1}{2}\sum_{k=1}^s(1-\langle {\bf d}^k,{\bf d}^k\rangle) m_k-\frac{1}{2}\sum_{k=1}^sm_k+1.$$

Next we estimate the (co-)dimension of the Luna stratum $S_\xi$. Inside the product of moduli spaces of stable representations $$\prod_{k=1}^sM_{{\bf d}^k}^{\Theta-\rm st}(Q),$$
we have the open subset $Y$ of tuples $(U_1,\ldots,U_s)$ of pairwise non-isomorphic stable representations. Mapping such a tuple  to the representation $\bigoplus_{k=1}^sU_k^{m_k}$ defines a surjective map $Y\rightarrow S_\xi$, thus $$\dim S_\xi\leq\dim Y=\sum_{k=1}^s\dim M_{{\bf d}^k}^{\Theta-\rm st}(Q)=\sum_{k=1}^s(1-\langle{\bf d}^k,{\bf d}^k\rangle)$$
since if $S_\xi\not=\emptyset$, then $M_{{\bf d}^k}^{\Theta-\rm st}(Q)\not=\emptyset$ for all $k$. Thus
$${\rm codim}_{M_{\bf d}^{\Theta-\rm sst}(Q)}S_\xi\geq 1-\langle{\bf d},{\bf d}\rangle-\sum_{k=1}^s(1-\langle{\bf d}^k,{\bf d}^k\rangle).$$

Combining these two estimates, we have
$$\dim p^{-1}(V)-\frac{1}{2}{\rm codim}_{M_{\bf d}^{\Theta-\rm sst}(Q)}S_\xi\leq$$
$$\leq -\frac{1}{2}\sum_{k=1}^s(1-\langle{\bf d}^k,{\bf d}^k\rangle)(m_k-1)-\frac{1}{2}(\sum_{k=1}^sm_k-1)\leq 0$$
since $1-\langle{\bf d}^k,{\bf d}^k\rangle=\dim M_{{\bf d}^k}^{\Theta-\rm st}(Q)\geq 0$ and $m_k\geq 1$ for all $k$, with equality holding if and only if $\xi$ is trivial. This finishes the proof of smallness.

\end{proof}

\subsection{Consequences}\label{s43}

As an immediate corollary to the smallness of the desingularization map $p$, we find:

\begin{corollary} Keep the assumptions of Theorem \ref{maintheorem}. Then we have the following formula for the compactly supported intersection cohomology of the moduli space $M_{\bf d}^{\Theta-\rm sst}(Q)$:

$$\sum_i\dim{\rm IH}_c^i(M_{\bf d}^{\Theta-\rm sst}(Q),{\bf Q})(-q^{1/2})^i=$$
$$=(q-1)\sum_{{\bf d}^*}(-1)^{s-1}q^{-\sum_{k\leq l}\langle{\bf d}^l,{\bf d}^k\rangle}\prod_{k=1}^s\prod_{i\in Q_0}\prod_{j=1}^{d^k_i}(1-q^{-j})^{-1},$$
where the sum ranges over all ordered decompositions ${\bf d}={\bf d}^1+\ldots+{\bf d}^s$ into non-zero dimension vectors such that
$$\Theta'({\bf d}^1+\ldots+{\bf d}^k)>0$$
for all $k<s$.\\[1ex]
In particular, the odd intersection cohomology of $M_{\bf d}^{\Theta-\rm sst}(Q)$ vanishes.
\end{corollary}

\proof If $f:X\rightarrow Y$ is a small desingularization, we have ${\rm IH}_c^*(Y,{\bf Q})\simeq H_c^*(X,{\bf Q})$. Applying this to the desingularization map $p$ and using the formula in Theorem \ref{bettiformula} for the Betti numbers of moduli spaces of semistable representations in the coprime case, the claimed formula follows.\\[2ex]
An explanation for the existence of our small resolutions in terms of Donaldson-Thomas theory can be derived from an observation in \cite{MICRA}: whenever the Euler form is symmetric on ${\rm Ker}(\Theta)$ and $\Theta'$ is a generic deformation of $\Theta$, we have ${\rm DT}_{\bf d}^{\Theta'}(Q)={\rm DT}_{\bf d}^{\Theta}(Q)$, and both invariants compute intersection cohomology by Theorem \ref{intersectionbettiformula}.

\section{Numerical aspects}\label{s5}

In this section, we discuss in which generality our main result is applicable. We first prove that generic deformations of stabilities, with respect to an indivisible dimension vector, always exist. Then we analyze the class of quivers and stabilities for which the restriction of the Euler form to the kernel of the stability function is symmetric: it turns out that such a stability exists whenever the antisymmetrized Euler form of the quiver has rank at most two.

\begin{theorem}\label{theorem51} For an indivisible dimension vector ${\bf d}$ and a stability $\Theta$ such that $\Theta({\bf d})=0$, there exists a generic deformation of $\Theta$.
\end{theorem}

\proof Since ${\bf d}$ is indivisible, there exists a stability $\eta$ such that $\eta({\bf d})=0$ and $\eta({\bf e})\not=0$ whenever $0\not={\bf e}\leq{\bf d}$ is such that $\Theta({\bf e})=0$. We choose $C\in{\bf N}$ such that $$C>\max(\max(\eta({\bf e})\, :\, {\bf e}\leq{\bf d},\, \Theta({\bf e})<0),\max(-\eta({\bf e})\, :\, {\bf e}\leq {\bf d},\, \Theta({\bf e})>0)).$$ 
We claim that $\Theta'=C\Theta+\eta$ is a generic deformation of $\Theta$:\\[1ex]
Suppose first that, for a dimension vector $0\not={\bf e}\lneqq{\bf d}$, we have $\Theta({\bf e})<0$, thus $\Theta({\bf e})\leq -1$. Then $$\Theta'({\bf e})=C\Theta({\bf e})+\eta({\bf e})\leq -C+\eta({\bf e})<0$$
by the choice of $C$. Now suppose that $\Theta'({\bf e})=0$, thus $C\Theta({\bf e})+\eta({\bf e})=0$, that is, $\eta({\bf e})=-C\Theta({\bf e})$. Then
$$C>|\eta({\bf e})|=C|\Theta({\bf e})|,$$
which implies $\Theta({\bf e})=0$. But then $\eta({\bf e})=0$, a contradiction to the choice of $\eta$. Finally, suppose that $\Theta'({\bf e})< 0$; we want to prove that $\Theta({\bf e})\leq 0$. Suppose, to the contrary, that $\Theta({\bf e})\geq 1$. Then
$$0\geq\Theta'({\bf e})=C\Theta({\bf e})+\eta({\bf e})\geq C+\eta({\bf e}),$$
that is, $\eta({\bf e})\leq -C$, in contradiction to the choice of $C$. The claim is proved.\\[1ex]
The following statement is closely related to a result of \cite[Section 2]{Stienstra}:

\begin{proposition} The restriction of the Euler form $\langle\_,\_\rangle$ of $Q$ to ${\rm Ker}(\Theta)$ is symmetric if and only if there exists a linear form $\eta\in\Lambda_{\bf Q}^*$ such that the antisymmetrized Euler form $\{\_,\_\}$ is given by
$$\{{\bf d},{\bf e}\}=\eta({\bf d})\Theta({\bf e})-\Theta({\bf d})\eta({\bf e})$$
for all ${\bf d},{\bf e}\in\Lambda^+$.
\end{proposition}

\proof Assume that the antisymmetrized Euler form vanishes on ${\rm Ker}(\Theta)$. If $\Theta=0$, there is nothing to prove. For $\Theta=\sum_i\Theta_i i^*\not=0$ we can assume, without loss of generality, that ${\rm gcd}_i(\Theta_i)=1$. Choose a ${\bf Q}$-basis $v_1,\ldots,v_{n-1}$ of ${\rm Ker}(\Theta)\subset\Lambda_{\bf Q}$, and extend it to a ${\bf Q}$-basis of $\Lambda_{\bf Q}$ by a vector $v_n$ such that $\Theta(v_n)=1$. We define $\eta\in\Lambda_{\bf Q}^*$ by $\eta(v_i)=\{v_i,v_n\}$ for $i=1,\ldots,n$. It is immediately verified that 
$$\{{v_i},{ v_j}\}=\eta({ v_i})\Theta({ v_j})-\Theta({ v_i})\eta({v_j})$$
for all $i,j=1,\ldots,n$, thus the claimed identity holds for all ${\bf d},{\bf e}\in\Lambda$. The converse implication is trivial.

\begin{corollary} For a given quiver $Q$, there exists a stability $\Theta\in\Lambda^*$ such that the Euler form of $Q$ is symmetric on ${\rm Ker}(\Theta)$ if and only if the antisymmetrized Euler form of $Q$ is of rank at most two. In particular, this is the case if $Q$ is a symmetric quiver, a complete bipartite quiver, or a quiver with at most three vertices.
\end{corollary}

\section{Examples}\label{s6}

In this final section, we discuss various classes of GIT quotients to which our methods apply.

\subsection{Classical invariant theory and determinantal varieties}\label{s61}

Consider the symmetric quiver $Q$ with two vertices $i$ and $j$ and $m\geq 1$ arrows from $i$ to $j$ as well as from $j$ to $i$. Let ${\bf d}=i+rj$ for $0\leq r\leq m$, and consider $\Theta=0$.\\[1ex]
A representation of $Q$ of dimension vector ${\bf d}$ is thus given by an $m$-tuple of vectors $v_1,\ldots,v_m$ in ${\bf C}^d$, together with an $m$-tuple of covectors $\varphi_1,\ldots,\varphi_m\in({\bf C}^d)^*$. On this data, the group ${\rm GL}_d({\bf C})$ acts via the natural action, and an additional ${\bf C}^*$ acts by dilation. By classical invariant theory, the ring of invariants is thus generated by the functions $\varphi_i(v_j)$ for $i,j=1,\ldots,m$. More geometrically, we have
$$M_{\bf d}^{\rm ssimp}(Q)\simeq M_{m\times m}^{\leq r}({\bf C})=\{A\in M_{m\times m}({\bf C})\, :\, {\rm Rank}(A)\leq r\},$$
the determinantal variety of $m\times m$-matrices of rank at most $r$. As long as $r<m$, this is a singular affine variety, with singular locus consisting of the matrices of rank less than $r$.\\[1ex]
It is obvious from the definitions that $\Theta'=ri^*-j^*$ defines a generic deformation of $\Theta=0$. The ring of semi-invariants with respect to this stability is generated, over the ring invariants, by the maximal minors of the $r\times m$-matrix $[v_1|\ldots|v_m]$. Viewing these as the Pluecker coordinates for the subspace $\langle v_1,\ldots,v_m\rangle\subset {\bf C}^r$, we see that
$$M_{\bf d}^{\Theta'-sst}(Q)\simeq Y=\{(A,U)\subset M_{m\times m}({\bf C})\times{\rm Gr}_r({\bf C}^m)\, :\, {\rm Im}(A)\subset U\}.$$
The desingularization $p$ then identifies with the canonical projection $Y\rightarrow M_{m\times m}^{\leq r}({\bf C})$, which is thus small.\\[1ex]
The ${\rm GL}_m({\bf C})$-equivariant canonical projection $Y\rightarrow {\rm Gr}_r({\bf C}^m)$ realizes $Y$ as a homogeneous bundle of rank $mr$ over ${\rm Gr}_r({\bf C}^m)$, thus we arrive at an identification
$${\rm IC}_c^\bullet(M_{m\times m}({\bf C})^{\leq r})\simeq H^\bullet({\rm Gr}_r({\bf C}^m))[2mr].$$

\subsection{Ordered point configurations in projective space}\label{s62}

Consider the quiver $Q$ with set of vertices $\{i_1,\ldots,i_m,j\}$ and one arrow from $i_k$ to $j$ for $k=1,\ldots,m$. Let the dimension vector ${\bf d}$ be given by $d_{i_k}=1$ for all $k$ and $d_{j}=d\geq 2$. We consider the stability respecting the natural $S_m$-symmetry, given as $\Theta=d\sum_{k=1}^mi_k^*-mj^*$.\\[1ex]
A representation of $Q$ of dimension vector ${\bf d}$ thus consists of a tuple $(v_1,\ldots,v_m)$ of vectors in ${\bf C}^d$, which are considered up to the natural diagonal action of ${\rm GL}_d({\bf C})$ and up to the action of an $m$-torus scaling the vectors.\\[1ex]
Such a tuple $(v_1,\ldots,v_m)$ is $\Theta$-semistable if, for every subset $I$ of $\{1,\ldots,m\}$, the span
$$U_I=\langle v_k\, :\, k\in I\rangle$$
fulfills the dimension estimate
$$\dim U_I\geq\frac{d}{m}|I|.$$

The tuple is stable if this estimate is strict for every non-empty proper subset $I$.\\[1ex]
The moduli space $M_{\bf d}^{\Theta-\rm sst}(Q)$ can thus be identified with the moduli space
$$({\bf P}^{d-1})^m_{\rm sst}//{\rm PGL}_d({\bf C})$$
of stable ordered configurations of $m$ points in ${\bf P}^{d-1}$, where stability is defined by the above dimension estimates.\\[1ex]
To find a generic deformation of the stability $\Theta$, let $g={\rm gcd}(d,m)$ and write $d=ge$, $m=gn$. For a subset $I\subset\{1,\ldots,m\}$ and $l\leq d$, define the dimension vector ${\bf d}(I,l)=\sum_{k\in I}i_k+lj$, thus $\Theta({\bf d}(I,l))=d|I|-ml$, which allows to determine the set of dimension vectors ${\bf e}\leq{\bf d}$ for which $\Theta({\bf e})>0$ or $\Theta({\bf e})\geq 0$. We can then exhibit a deformed stability for example by applying the construction in the proof of Theorem \ref{theorem51} with $\eta=di_1^*-j^*$ and $C=d$. Thus
$$\Theta'=(d^2+d)i_1^*+ \sum_{k=2}^md^2i_k^*-(md+1)j^*$$
is a generic deformation of $\Theta$. A tuple $(v_1,\ldots,v_m)$ is then $\Theta'$-(semi-)stable if and only if the previous dimension estimate $\dim U_I\geq d/m|I|$ is fulfilled for every subset $I$, and it is strict whenever $1\in I$.\\[1ex]
To determine the possible local quivers $Q_\xi$, we note that the dimension vectors ${\bf e}\leq{\bf d}$ of subrepresentations fulfilling $\Theta({\bf e})=0$ are those of the form ${\bf e}={\bf d}(I,\frac{e}{n}|I|)$, where $I$ is a subset of $\{1,\ldots,m\}$ whose cardinality is divisible by $n$. The possible decomposition types of ${\bf d}$ are therefore parametrized by decompositions of $\{1,\ldots,m\}$ into pairwise disjoint subsets $I_1,\ldots,I_s$ all of whose cardinalities are divisible by $n$. We associate to this decomposition type a marked partition $\dot{\lambda}$ of $g$ with parts $|I_k|/n$ for $k=1,\ldots,s$ and marked part $|I_1|/n$.\\[1ex]
It is then easy to see that the corresponding local quivers only depend on the marked partition: $Q_{\dot{\lambda}}$ has vertices $v_1,\ldots,v_s$, and the number of arrows from $v_p$ to $v_q$ equals $e(e-n)\lambda_p\lambda_q$; the stability $\Theta_{\dot{\lambda}}$ is defined by $\Theta_{\dot{\lambda}}(\lambda_p)=d-{e}\lambda_p$ for the marked part $\lambda_p$ and $\Theta_{\dot{\lambda}}(\lambda_q)=-e\lambda_q$ for all other parts $\lambda_q$. For this quiver with stability, we consider the dimension vector ${\bf 1}=\sum_pv_p$.\\[1ex]
We conclude that all fibres of the desingularization map are of the form
$$M_{{\bf 1}}^{\Theta_{\dot{\lambda}}-\rm sst}(Q_{\dot{\lambda}}).$$

\subsection{Levi invariants in adjoint actions}\label{s63}

Consider a finite-dimensional complex vector spaces $V$ with a fixed direct sum decomposition $V=\bigoplus_{p=1}^lV_p$. Consider the adjoint action of ${\rm GL}(V)$ on its Lie algebra $\mathfrak{gl}(V)$ and the Levi subgroup $L=\prod_{p=1}^{l}{\rm GL}(V_p)$ of ${\rm GL}(V)$. We are interested in the action of $L$ on $\mathfrak{gl}(V)$ given by restriction of the adjoint action.\\[1ex]
Define $Q$ as the quiver with $l=l(\lambda)$ vertices $i_1,\ldots,i_{l}$, and one arrow from $i_p$ to $i_q$ for all $p,q=1,\ldots,l$. Define the dimension vector ${\bf d}=\sum_{p=1}^l\dim V_p i_p$, and consider the trivial stability $\Theta=0$. Then $$M_{\bf d}^{\rm ssimp}(Q)\simeq \mathfrak{gl}(V)//L.$$

To obtain a small desingularisation, we restrict to the case ${\rm gcd}_p(\dim V_p)=1$. A stability $\Theta'$ such that $\sum_p\Theta'_p\dim V_p=0$ is generic if and only if $\sum_p\Theta'_p a_p\not=0$ for every tuple $(0\leq a_p\leq\dim V_p)_p$ which is neither identically zero nor equal to $(\dim V_p)_p$. For every such $\Theta'$, the moduli space $M_{\bf d}^{\Theta'-\rm sst}(Q)$ yields a small desingularization of $\mathfrak{gl}(V)//L$.\\[1ex]
We consider the particular case $\dim V_p=1$ for all $p$. Then $\Theta'$ is a generic stability if and only if $\sum_p\Theta'_p=0$ and $\sum_{p\in I}\Theta'_p\not=0$ for every proper non-empty subset $I\subset\{1,\ldots,l\}$. A particular choice of such a stability is provided by, for example, $\Theta'=(l-1)i_1^*-\sum_{p=2}^li_p^*$.\\[1ex]
We are thus considering the quotient of $\mathfrak{gl}(V)$ by a maximal torus $T$ of ${\rm GL}(V)$. We fix a basis consisting of generators of the $V_p$ and identify $\mathfrak{gl}(V)$ with $\mathfrak{gl}_n({\bf C})$ and $L$ with the group of diagonal matrices. A matrix $A\in\mathfrak{gl}_n({\bf C})$ is then $\Theta'$-semistable if and only if, for every $p=1,\ldots,n$ there exists a sequence $1=p_1,p_2,\ldots,p_s=p$ such that the matrix elements
$$A_{p_1,p_2},A_{p_2,p_3},\ldots,A_{p_{s-1},p_s}$$
are all non-zero. In other words, defining the support graph of $A$ as the oriented graph on $\{1,\ldots,n\}$ with arrows $p\rightarrow q$ whenever $A_{p,q}\not=0$, a matrix $A$ is $\Theta'$-semistable if and only if its support graph contains a spanning tree with root $1$.\\[1ex]
As a particular example, consider the case $l=3$, and write
$$A=\left[\begin{array}{lll}a&b&c\\ d&e&f\\ g&h&i\end{array}\right]\in\mathfrak{gl}_3({\bf C}).$$
Then the invariant ring is generated by
$$a,e,i,x=bd,y=cg,z=fh,v=bfg,w=cdh,$$
with the single defining relation $vw=xyz$; the moduli space is thus a cubic hypersurface in ${\bf A}^8$. The desingularization is given by affine cordinates $a,e,i,x,y,z,v,w$ and homogeneous coordinates $F_0,F_1,F_2$, subject to the relations
$$F_0v=F_1z,\; F_0w=F_2=x,\; F_1F_2=F_0^2y.$$
This is a rank $5$ bundle over ${\bf P}^2$, so the compactly supported intersection cohomology of $\mathfrak{gl}_3({\bf C})//({\bf C}^*)^3$ is isomorphic to $H^\bullet({\bf P}^2)[10]$.

\subsection{Levi invariants for graded linear maps}\label{s64}

Similarly to the previous example, fix two finite-dimensional complex vector spaces $V$ and $W$ with direct sum decompositions $V=\bigoplus_{p=1}^lV_p$, $W=\bigoplus_{q=1}^lW_q$. Consider the natural action of ${\rm GL}(V)\times{\rm GL}(W)$ on ${\rm Hom}(V,W)$ and the Levi subgroup $L=\prod_p{\rm GL}(V_p)\times\prod_q{\rm GL}(W_q)\subset {\rm GL}(V)\times{\rm GL}(W)$. We are interested in the action of $L$ on ${\rm Hom}(V,W)$ given by restriction of the above action.\\[1ex]
Let $Q$ be the complete bipartite quiver with set of vertices $\{i_1,\ldots,i_k,j_1,\ldots,j_l\}$ and arrows $\alpha_{p,q}:i_p\rightarrow j_q$ for $p=1,\ldots,k$, $q=1,\ldots,l$. We define a dimension vector ${\bf d}$ by $${\bf d}=\sum_{p=1}^k\dim V_p i_p+\sum_{q=1}^l\dim W_qj_q.$$

A representation of $Q$ of dimension vector ${\bf d}$ thus consists of a collections of linear maps $f_{p,q}:V_p\rightarrow W_q$ for $p=1,\ldots,k$, $q=1,\ldots,l$, which can be assembled into the element
$$\bigoplus_{p,q}f_{p,q}\in{\rm Hom}(V,W).$$

Let $\Theta$ be the stability given by $\Theta(i_p)=\dim W$ for all $p$ and $\Theta(j_q)=-\dim V$ for all $q$. A representation given by linear maps $(f_{p,q})_{p,q}$ is then $\Theta$-semistable (resp.~$\Theta$-stable) if and only if, for all non-trivial tuples of subspaces $(U_p\subset V_p)_p$, we have
$$\sum_q\dim\sum_pf_{p,q}U_p\geq\frac{\dim W}{\dim V}\sum_p\dim U_p$$
(resp.~strict inequality).\\[1ex]
Defining ${\rm Hom}(V,W)^{\rm (s)st}$ as the corresponding locus in ${\rm Hom}(V,W)$, we have thus, by definition, an identification
$$M_{\bf d}^{\Theta-\rm sst}(Q)\simeq{\rm Hom}(V,W)^{\rm sst}//L.$$

The assumption that the restriction of the Euler form to the kernel of the stability is symmetric is fulfilled: the antisymmetrized Euler form of $Q$ is given by
$$\{{\bf e},{\bf f}\}=(\sum_pf_{i_p})(\sum_q e_{j_q})-(\sum_{p}e_{i_p})(\sum_q f_{j_q}),$$
thus vanishes for ${\bf e},{\bf f}\in{\rm Ker}(\Theta)$ since
$$\frac{\sum_pe_{i_p}}{\sum_qe_{j_q}}=\frac{\dim V}{\dim W}=\frac{\sum_pf_{i_p}}{\sum_q f_{j_q}}.$$

We have thus proved: if ${\rm gcd}((\dim V_p)_p,(\dim W_q)_q)=1$ and ${\rm Hom}(V,W)^{\rm st}\not=\emptyset$, the quotient ${\rm Hom}(V,W)^{\rm sst}//L$ admits a small desingularization.

\subsection{Toric quiver moduli and abelian invariants}\label{s65}

We briefly review some concepts of \cite{MoR}. Suppose again that we are given a quiver $Q$, an arbitrary dimension vector ${\bf d}$ and a stability $\Theta$ such that $\Theta({\bf d})=0$ and such that the restriction of the Euler form to ${\rm Ker}(\Theta)$ is symmetric. We construct a new quiver $Q_{\bf d}$ with set of vertices
$$\{i_k\, :\, i\in Q_0,\, k=1,\ldots,d_i\}$$
and arrows
$$\alpha_{k,l}:i_k\rightarrow j_l\mbox{ for }\alpha:i\rightarrow j\mbox{ in }Q,\, k=1,\ldots,d_i,\, l=1,\ldots,d_j\}.$$
We define a stability $\Theta_{\bf d}$ for $Q_{\bf d}$ by $\Theta_{\bf d}(i_k)=\Theta(i)$ for all $i\in Q_0$ and $k=1,\ldots,d_i$. We consider the dimension vector ${\bf 1}$ for $Q_{\bf d}$ given as $${\bf 1}=\sum_{i\in Q_0}\sum_{k=1}^{d_i}i_k.$$
Then the moduli spaces
$$M_{\bf 1}^{\Theta_{\bf d}-\rm sst}(Q_{\bf d})$$
are toric varieties since the torus action scaling all arrows of $Q_{\bf d}$ acts with a dense orbit.\\[1ex]
They are used in \cite{MoR} for the definition of certain abelian analogues of quantized Donaldson-Thomas invariants $g_{\bf d}^{\Theta}(Q)$, which (in the case of generalized Kronecker quivers) play the role of universal coefficients in the MPS wall-crossing formula of \cite{MPS}.\\[1ex]
In \cite[Section 6]{MoR}, these invariants are identified with the Poincar\'e polynomials of certain desingularizations of $M_{\bf 1}^{\Theta_{\bf d}-\rm sst}(Q_{\bf d})$. These actually coincide with the desingularizations constructed in Section \ref{s4}, thus we find: the abelian DT invariant $g_{\bf d}^\Theta(Q)$ coincides (up to a shift) with the Poincar\'e polynomial in compactly supported intersection cohomology of $M_{\bf 1}^{\Theta_{\bf d}-\rm sst}(Q_{\bf d})$.

\subsection{Counterexamples}\label{s66}

We exhibit counterexamples showing that the assumptions of the main theorem are essential:\\[1ex]
Consider an arbitrary symmetric quiver $Q$ with two dimension vectors ${\bf e}$, ${\bf f}$ such that the following holds: there exist simple representations of dimension vectors ${\bf e}$ and ${\bf f}$, there does not exist a simple representation of dimension vector ${\bf e}+{\bf f}$, and $\langle{\bf e},{\bf f}\rangle=0$. Then a general representation of dimension vector ${\bf e}+{\bf f}$ is isomorphic to the direct sum $V\oplus W$, where ${\bf dim} V={\bf e}$ and ${\bf dim} W={\bf f}$. Thus any such representation $V\oplus W$ is unstable for every choice of stability $\Theta$ such that $\Theta({\bf e})\not=0$, which implies $M_{{\bf e}+{\bf f}}^{\Theta-\rm sst}(Q)=\emptyset$, although $M_{{\bf e}+{\bf f}}^{\rm ssimp}(Q)\not=\emptyset$. As a concrete example, one can take the quiver $Q$ given as
$$i{{\rightarrow}\atop{\leftarrow}}j{{\rightarrow}\atop{\leftarrow}}k$$
and the dimension vectors ${\bf e}=i+j+k$ and ${\bf f}=k$.\\[1ex]
Finally we exhibit an example showing that the symmetry of the Euler form (restricted to the kernel of the stability) is essential. In generalization of the first example in Section \ref{s61}, consider the quiver with two vertices $i$ and $j$, $m$ arrows from $i$ to $j$, and $n$ arrows from $j$ to $i$. We consider the dimension vector ${\bf d}=i+j$ and the stabilites $\Theta=0$, $\Theta'=i^*-j^*$. Again we find
$$M_{\bf d}^{\rm ssimp}(Q)\simeq\{A\in M_{m\times n}({\bf C})\, :\, {\rm Rank}(A)\leq 1\}$$
and
$$M_{\bf d}^{\Theta'-\rm sst}(Q)\simeq\{(A,L)\in M_{m\times n}({\bf C})\times{\bf P}^{m-1}\, :\, {\rm Im}(A)\subset L\}.$$
The fibre of the desingularization map over the zero matrix is the projective space ${\bf P}^{m-1}$, and the dimension of the variety of matrices of rank at most one equals $m+n-1$. Thus the desingularization is not small as soon as $m>n$.

\end{document}